\documentclass[a4paper,12pt]{amsart}
\usepackage{amssymb}
\usepackage{ifthen}
 \usepackage[dvips]{graphicx}
\nonstopmode \numberwithin{equation}{section}
\usepackage{amssymb}
\usepackage{ifthen}
\usepackage{subcaption}
\usepackage{graphicx}
\usepackage[margin=1in]{geometry}
\usepackage{cite}
\usepackage{amsmath}
\usepackage[T1]{fontenc} %skandit
\usepackage[utf8]{inputenc}
\usepackage[usenames,dvipsnames]{color}
\usepackage{color}
\usepackage[english]{babel}
\usepackage{fancyhdr}
\usepackage{fancybox}
\usepackage{tikz}
\nonstopmode\numberwithin{equation}{section}
\newtheorem*{thmA}{Theorem A}
\newtheorem*{thmB}{Theorem B}
\newtheorem*{thmC}{Theorem C}
\newtheorem*{thmD}{Theorem D}

\newtheorem*{lemA}{Lemma A}
\newtheorem*{lemB}{Lemma B}
\newtheorem*{lemC}{Lemma C}
\newtheorem*{lemD}{Lemma D}

\theoremstyle{plain}

\newtheorem{conj}{Conjecture}

\theoremstyle{definition}
\newtheorem{defn}{Definition}[section]

\newtheorem{thm}{Theorem}[section]
\newtheorem{prob}{Problem}[section]
\newtheorem{cor}{Corollary}[section]

\newtheorem{prop}{Proposition}[section]
\newtheorem{rem}{Remark}[section]
\newtheorem{lem}{Lemma}[section]

\newcounter{minutes}
\divide\time by 60
\newcounter{hours}
\multiply\time by 60
\addtocounter{minutes}{-\time}
\newcounter {own}
\def\theown {\thesection       .\arabic{own}}

\newenvironment{pf}[1][]{%
 \vskip 3mm
 \noindent
 \ifthenelse{\equal{#1}{}}%
  {{\slshape Proof. }}%
  {{\slshape #1.} }%
 }%
{\qed\bigskip}

\newcounter{alphabet}

\def\be{\begin{equation}}
\def\ee{\end{equation}}

\newcommand{\bee}{\begin{enumerate}}
\newcommand{\eee}{\end{enumerate}}

\newcommand{\blem}{\begin{lem}}
\newcommand{\elem}{\end{lem}}
\newcommand{\bthm}{\begin{thm}}
\newcommand{\ethm}{\end{thm}}
\newcommand{\bcor}{\begin{cor}}
\newcommand{\ecor}{\end{cor}}
\newcommand{\beg}{\begin{examp}}
\newcommand{\eeg}{\end{examp}}
\newcommand{\begs}{\begin{examples}}
\newcommand{\eegs}{\end{examples}}

\newcommand{\bdefn}{\begin{defn}}
\newcommand{\edefn}{\end{defn}}

\newcommand{\bprob}{\begin{prob}}
\newcommand{\eprob}{\end{prob}}
\newcommand{\bei}{\begin{itemize}}
\newcommand{\eei}{\end{itemize}}

\newcommand{\bcon}{\begin{conj}}
\newcommand{\econ}{\end{conj}}
\newcommand{\bcons}{\begin{conjs}}
\newcommand{\econs}{\end{conjs}}
\newcommand{\bprop}{\begin{prop}}
\newcommand{\eprop}{\end{prop}}
\newcommand{\br}{\begin{rem}}
\newcommand{\er}{\end{rem}}
\newcommand{\brs}{\begin{rems}}
\newcommand{\ers}{\end{rems}}
\newcommand{\bo}{\begin{obser}}
\newcommand{\eo}{\end{obser}}
\newcommand{\bos}{\begin{obsers}}
\newcommand{\eos}{\end{obsers}}
\newcommand{\bpf}{\begin{pf}}
\newcommand{\epf}{\end{pf}}
\newcommand{\ba}{\begin{array}}
\newcommand{\ea}{\end{array}}
\newcommand{\beq}{\begin{eqnarray}}
\newcommand{\beqq}{\begin{eqnarray*}}
\newcommand{\eeq}{\end{eqnarray}}
\newcommand{\eeqq}{\end{eqnarray*}}

\begin{document}

\title{Sharp Bohr-Type Inequalities Involving Euler Operator and Area Functionals on $\mathbb{P}\Delta(0;1_n)$}

\author{Molla Basir Ahamed}
\address{Molla Basir Ahamed, Department of Mathematics, Jadavpur University, Kolkata-700032, West Bengal, India.}
\email{mbahamed.math@jadavpuruniversity.in}

\author{Taimur Rahman}
\address{Taimur Rahman, Department of Mathematics, Jadavpur University, Kolkata-700032, West Bengal, India.}
\email{taimurr.math.rs@jadavpuruniversity.in}

\subjclass[{AMS} Subject Classification:]{Primary 32A05, 32A70}

\def\thefootnote{}
\footnotetext{ {\tiny File:~\jobname.tex,
printed: \number\year-\number\month-\number\day,
          \thehours.\ifnum\theminutes<10{0}\fi\theminutes }
} \makeatletter\def\thefootnote{\@arabic\c@footnote}\makeatother

\begin{abstract}
	In this paper, we establish several higher-dimensional generalizations of refined Bohr-type inequalities for bounded holomorphic functions mapping into the unit polydisk $\mathbb{P}\Delta(0;1_n)$ in $\mathbb{C}^n$. First, we formulate multidimensional analogues of sharp Bohr-type inequalities originally established by Liu \emph{et al.} [{\it Bull. Sci. Math.} {\bf 173} (2021) 103054], incorporating both squared coefficient terms and area functional components. Second, we provide improved inequalities for a recent multidimensional extension by Ahamed \emph{et al.} [{\it Complex Anal. Oper. Theory} {\bf 20}(6) (2026), 142] by introducing an analogous term corresponding to the area functional. Finally, we extend a refined Bohr-type inequality involving the term $\vert{}f(z)-a_0\vert{}$ to the setting of several complex variables. All the results are shown to be sharp.
\end{abstract}

\maketitle
\pagestyle{myheadings}
\markboth{M. B. Ahamed and T. Rahman}{Multidimensional Bohr phenomenon on unit polydisc}
\tableofcontents
\section{\bf Introduction}
The classical theorem of Harald Bohr \cite{Bohr-PLMS-1914}, originally examined a century ago, continues to generate intensive research on what is now known as Bohr’s phenomenon. Renewed interest surged in the $1990$s due to successful extensions to holomorphic functions of several complex variables and to more abstract functional analytic settings. In $1997$, Boas and Khavinson \cite{Boas-Khavinson-PAMS-1997} extended the Bohr phenomenon to higher dimensions by determining the $n$-dimensional Bohr radius on the polydisk. This seminal work stimulated significant research interest in Bohr-type questions across diverse mathematical domains.Subsequent investigations have yielded further results on Bohr’s phenomenon for multidimensional power series, leading to major contributions by Aizenberg \cite{Aizenberg-2005,Aizenberg-PAMS-2000,Aizenberg-StudMath-2007}, Aizenberg \textit{et al.} \cite{Aizenberg-Aytuna-Djakov-JMAA-2001,Aizenberg-Elin-Shoikhet-StudMath-2005}, Defant and Frerick \cite{Defant-Frerick-IJM-2006}, and Djakov and Ramanujan \cite{Djakov-Ramanujan-JA-2000}. Kumar and Ponnusamy \cite{Kumar-Ponnusamy-JGA-2026} obtained bounds on the partial derivatives at the origin for pluriharmonic functions mapping the unit ball of Minkowski space into the Euclidean unit ball in $\mathbb{C}^N$. Furthermore, they determined the exact asymptotic values of the mixed and arithmetic Bohr radii for these functions. For comprehensive reviews of the topic, we refer to \cite{Ahamed-Ahammed-MJM-2024,Ahamed-Allu-BMMSS-2022,Ahamed-Roy-PMS-2025,Alkhaleefah-Kayumov-Ponnusamy-PAMS-2019,Defant-Frerick-OrtegaCerda-Ounaies-Seip-AM-2011,Hamada-Honda-Kohr-IM-2009,Liu-Ponnusamy-PAMS-2021,Lata-Singh-PAMS-2022,Paulsen-Popescu-Singh-PLMS-2002,Paulsen-Singh-PAMS-2004,Paulsen-Singh-BLMS-2006,Kumar-Manna-JMAA-2023,Kumar-PAMS-2023} and the monograph by Kresin and Maz’ya \cite{Kresin-Mazya-2007}.
\subsection{\bf Overview of Bohr’s Inequality and Related Concepts}

Let $\mathcal{B}$ denote the class of analytic functions $f(\zeta) = \sum_{k=0}^{\infty} a_k\zeta^k$ mapping the unit disk $\mathbb{D} := \{\zeta \in \mathbb{C} : \vert{}\zeta\vert{} < 1\}$ into itself. Excluding trivial unimodular constants, every function in $\mathcal{B}$ satisfies $\vert{}f(\zeta)\vert{} < 1$ throughout $\mathbb{D}$. In $1914$, while investigating Dirichlet series, Harald Bohr \cite{Bohr-PLMS-1914} established the following foundational result
\begin{thmA}
	If $f \in \mathcal{B}$, then the following inequality holds 
	\begin{align}\label{eq-1.1}
		M_f(r) := \sum_{k=0}^{\infty} \vert{}a_k\vert{} r^k \le 1 \quad \text{for } \vert{}\zeta\vert{} := r \le \frac{1}{3}
	\end{align}
\end{thmA}
The radius $B=1/3$ is best possible and is referred to as the \textit{Bohr radius} and the inequality \eqref{eq-1.1} is known as \emph{Bohr inequality} for the class $\mathcal{B}$. Furthermore, for 
\begin{align*}
	\psi_a(\zeta) = \frac{a - \zeta}{1 - a\zeta}= a - (1 - a^2) \sum_{k=1}^{\infty} a^{k-1} \zeta^k, \quad a \in [0,1),\; \zeta\in\mathbb{D}
\end{align*} it follows easily that $M_{\psi_a}(r) > 1$ if, and only if, $r >{1}/{(1 + 2a)}$, and hence the radius $1/3$ is optimal as $a \to 1^-$. Bohr originally proved the inequality \eqref{eq-1.1} for $r \le 1/6$, whereas Wiener, Schur, and Riesz later showed that $1/3$ is the best possible. For other proofs of this theorem, we refer to the articles \cite{Bohr-PLMS-1914,Sidon-MZ-1927,Tomic-MS-1962}. It is worth pointing out that there is no extremal function in $\mathcal{B}$ such that the Bohr radius is precisely ${1/3}$ (see, e.g. \cite{Garcia-Mashreghi-Ross-2018}).\vspace{2mm}

Over the decades, a variety of alternative proofs have established (see, e.g. \cite{Garcia-Mashreghi-Ross-2018, Paulsen-Popescu-Singh-PLMS-2002, Paulsen-Singh-PAMS-2004, Paulsen-Singh-BLMS-2006, Sidon-MZ-1927, Tomic-MS-1962}), connecting Bohr's ideas with functional analysis, number theory, and probability. These techniques have also enabled higher-dimensional extensions of Bohr's result \cite{Aizenberg-2005, Aizenberg-Aytuna-Djakov-JMAA-2001, Aizenberg-PAMS-2000, Boas-Khavinson-PAMS-1997, Djakov-Ramanujan-JA-2000, Jia-Liu-Ponnusamy-AMP-2025}. It is interesting that substituting $\vert{}a_0\vert{}^2$ for $\vert{}a_0\vert{}$ in Bohr's inequality improves the radius from $1/3$ to $1/2$. Moreover, assuming $a_0 = 0$ increases the sharp radius to $1/\sqrt{2}$ (cf. \cite{Kayumov-Ponnusamy-CMFT-2017, Paulsen-Popescu-Singh-PLMS-2002, Ponnusamy-Wirths-CMFT-2020}). These refinements rely on the standard coefficient estimate $\vert{}a_n\vert{} \le 1 - \vert{}a_0\vert{}^2$ ($n \ge 1, f \in \mathcal{B}$). For in depth knowledge in Bohr radius proble, we refer to \cite{Muhanna-Ali-Ponnusamy-2016}. \vspace{1.5mm}

Let $\mathbb{N}$ be the set of positive integers. Rogosinski \cite{Rogosinski-MZ-1923} introduced a related phenomenon: for any function $f(\zeta) = \sum_{n=0}^{\infty} a_n \zeta^n$ satisfying $\vert{}f(\zeta)\vert{} < 1$ on $\mathbb{D}$, its partial sums $S_N(\zeta) = \sum_{n=0}^{N-1} a_n \zeta^n$ satisfy $\vert{}S_N(\zeta)\vert{} \le 1$ on the disk $\vert{}\zeta\vert{} \le R= 1/2$ for every $N \in \mathbb{N}$. The constant $1/2$ is sharp and is known as the Rogosinski radius. It is clear that the \textit{Bohr radius} $B$ and the \textit{Rogosinski radius} $R$ satisfies $B=1/3<1/2=R$. \vspace{2mm}

\subsection{\bf Improved Bohr-Type Inequalities}
Let $f$ be holomorphic in $\mathbb{D}$, and let $\mathbb{D}_r := \{\zeta \in \mathbb{D} : |\zeta| < r\}$ where $0 < r < 1$. The area of the Riemann surface corresponding to the image of $\mathbb{D}_r$ under $f$ is denoted by $S_r$ \cite{Kayumov-Ponnusamy-CRMASP-2018,Ismagilov-Kayumov-Ponnusamy-JMAA-2020}. The quantity $S_r$ is characterized by
\begin{align*}
	\frac{S_r}{\pi} = \frac{1}{\pi} \iint_{\mathbb{D}_r} |f'(\zeta)|^2 \, dx \, dy = \sum_{n=1}^{\infty} n |a_n|^2 r^{2n}.
\end{align*}
Furthermore, $S_r$ reduces to the area of $f(\mathbb{D}_r)$ when $f$ is univalent, while it strictly exceeds this area if $f$ is multivalent.\vspace{1.2mm}

In \cite{Liu-Liu-Ponnusamy-BDSM-2021}, Liu \emph{et al.} established the following improved Bohr inequality for functions in the class $\mathcal{B}$ in terms of the area functional $S_r/\pi$.
\begin{thmB}\cite[Theorem 4]{Liu-Liu-Ponnusamy-BDSM-2021}
	Suppose that $f \in \mathcal{B}$, $f(\zeta) = \sum_{n=0}^{\infty} a_n \zeta^n$, and $S_r$ denotes the Riemann surface of the function $f^{-1}$ defined on the image of the subdisk $|\zeta| < r$ under the mapping $f$. Then
	\begin{align*}
		\sum_{n=0}^{\infty} |a_n| r^n + \left( \frac{1}{1 + |a_0|} + \frac{r}{1 - r} \right) \sum_{n=1}^{\infty} |a_n|^2 r^{2n} + \frac{8}{9} \left( \frac{S_r}{\pi} \right) \le 1
	\end{align*}
	for $|\zeta|=r \le 1/3$, and the constants $1/3$ and $8/9$ cannot be improved. Moreover,
	\begin{align*}
		|a_0|^2 + \sum_{n=1}^{\infty} |a_n| r^n + \left( \frac{1}{1 + |a_0|} + \frac{r}{1 - r} \right) \sum_{n=1}^{\infty} |a_n|^2 r^{2n} + \frac{9}{8} \left( \frac{S_r}{\pi} \right) \le 1
	\end{align*}
	for $|\zeta|=r \le 1/(3 - a)$, and the constant $9/8$ cannot be improved, where $a = |a_0|$.
\end{thmB}

\begin{thmC}\cite[Theorem 6]{Liu-Liu-Ponnusamy-BDSM-2021}
	Suppose that $f \in \mathcal{B}$ and $f(\zeta) = \sum_{n=0}^{\infty} a_n \zeta^n$. Then
	\begin{align*}
		\sum_{n=0}^{\infty} |a_n| r^n + \left( \frac{1}{1 + |a_0|} + \frac{r}{1 - r} \right) \sum_{n=1}^{\infty} |a_n|^2 r^{2n} + |f(\zeta) - a_0| \le 1
	\end{align*}
	for $|\zeta| = r \le 1/5$ and the number $1/5$ cannot be improved. Moreover,
	\begin{align*}
		|a_0|^2 + \sum_{n=1}^{\infty} |a_n| r^n + \left( \frac{1}{1 + |a_0|} + \frac{r}{1 - r} \right) \sum_{n=1}^{\infty} |a_n|^2 r^{2n} + |f(\zeta) - a_0| \le 1
	\end{align*}
	for $|\zeta| = r \le 1/3$ and the constant $1/3$ cannot be improved.
\end{thmC}
Inspired by the recent findings of Ahamed \emph{et al.} \cite{Ahamed-Majumder-Sarkar-CAOT-2026}, it is natural to pose the following question.
\begin{prob}\label{pb-1.1}
	Is it possible to establish multidimensional versions of Theorem B and C?
\end{prob}
In this paper, one of our main objectives is to give an affirmative answer to Problem~\ref{pb-1.1}.
\subsection{\bf Basic notations in several complex variables} Let $z = (z_1, \dots, z_n)$ and $w = (w_1, \dots, w_n)$ be vectors in $\mathbb{C}^n$. The standard Hermitian inner product on $\mathbb{C}^n$ is given by $$\langle z, w \rangle = \sum_{j=1}^n z_j \bar{w}_j,$$ which induces the Euclidean norm $\Vert{}z\Vert{} = \sqrt{\langle z, z \rangle}$. The modulus of a single complex component is written as $\vert{}z_1\vert{}$, whereas the uniform norm on $\mathbb{C}^n$ is defined as $\Vert{}z\Vert{}_\infty = \max_{1 \le i \le n} \vert{}z_i\vert{}$. An open polydisk (alternatively termed an open polycylinder) centered at $a = (a_1, \dots, a_n) \in \mathbb{C}^n$ with polyradius $r = (r_1, \dots, r_n) \in \mathbb{R}^n(r_i>0)$ is the Cartesian product of $n$ one-dimensional open disks:
\begin{align*}
	\mathbb{P}\Delta(a; r) = \prod_{j=1}^n \Delta(a_j; r_j) = \left\{ z \in \mathbb{C}^n : \vert{}z_i - a_i\vert{} < r_i \text{ for all } i = 1, \dots, n \right\}.
\end{align*}
In particular, the standard unit polydisk centered at the origin is denoted by
\begin{align*}
	\mathbb{P}\Delta(\mathbf{0}; \mathbf{1}) =\mathbb{P}\Delta(\mathbf{0}; \mathbf{1}_n)= \prod_{j=1}^n \Delta(0_j; 1_j).
\end{align*}

The closure of $\mathbb{P}\Delta(a; r)$, denoted by $\overline{\mathbb{P}\Delta(a; r)}$, constitutes the closed polydisk. Let $C_k(a_k; r_k) = \partial \Delta(a_k; r_k)$ represent the boundary circle of radius $r_k$ in the $z_k$-plane, parametrized conventionally as $\theta_k \mapsto a_k + r_k e^{i\theta_k}$ for $\theta_k \in [0, 2\pi]$. The Cartesian product of these coordinate circles,
\begin{align*}
	C^n(a; r) := C_1(a_1; r_1) \times \dots \times C_n(a_n; r_n)
\end{align*}
is referred to as the determining set of the polydisk $\mathbb{P}\Delta(a; r)$.\vspace{2mm}

An $n$-dimensional multi-index $\alpha = (\alpha_1, \dots, \alpha_n)$ is an ordered tuple of non-negative integers $\alpha_j \ge 0$. The degree of $\alpha$ is given by $\vert{}\alpha\vert{} = \sum_{j=1}^n \alpha_j$, and its factorial is defined as $\alpha! = \prod_{j=1}^n \alpha_j!$. For any point $z \in \mathbb{C}^n$, we define the monomial powers $z^\alpha = \prod_{j=1}^n z_j^{\alpha_j}$ and $\vert{}z\vert{}^\alpha = \prod_{j=1}^n \vert{}z_j\vert{}^{\alpha_j}$. Let $f(z)$ be holomorphic in a domain $\Omega \subset \mathbb{C}^n$ containing the point $c \in \Omega$. Within any polydisk $\mathbb{P}\Delta(c; r) \subset \Omega$ centered at $c$, $f(z)$ can be represented as an absolutely convergent multi-variable Taylor series:
\begin{align*}
	f(z) = \sum_{\alpha_1,\dots,\alpha_n=0}^{\infty} a_\alpha (z_1 - c_1)^{\alpha_1} \cdots (z_n - c_n)^{\alpha_n} = \sum_{\vert{}\alpha\vert{}=0}^{\infty} a_\alpha (z - c)^\alpha = \sum_{k=0}^{\infty} P_k(z - c),
\end{align*}
where each $P_k(z - c) = \sum_{\vert{}\alpha\vert{}=k} a_\alpha (z - c)^\alpha$ denotes a homogeneous polynomial of degree $k$.\vspace{2mm}

 For a function $f$ holomorphic in a domain $\Omega \subset \mathbb{C}^n$ containing the origin, the Euler operator (or radial derivative operator) $D$ is defined by
 \begin{align*}
 	Df(z) := \sum_{k=1}^n z_k \frac{\partial f(z)}{\partial z_k}, \quad z = (z_1, \dots, z_n) \in \Omega.
 \end{align*}
 This linear differential operator plays a fundamental role in the theory of several complex variables, particularly in connection with homogeneous functions and the study of starlike mappings in $\mathbb{C}^n$. In the literature, $Df(z)$ is also frequently referred to as the radial derivative or total derivative of $f$ at $z$.\vspace{2mm}

Recently, Ahamed \emph{et al.} \cite{Ahamed-Majumder-Sarkar-CAOT-2026} proved the following sharp multidimensional results involving radial derivative on the polydisk.
\begin{thmD}\cite[Theorem 2.3]{Ahamed-Majumder-Sarkar-CAOT-2026}
	Let $f(z)=\sum_{|\alpha|=0}^{\infty}a_{\alpha}z^{\alpha}$ be a holomorphic function in the polydisk $\mathbb{P}\Delta(0;1_n)$ such that $|f(z)|\leq 1$ for all \(z\in\mathbb{P}\Delta(0;1/n)\). Suppose $z=(z_{1},\ldots,z_{n})\in\mathbb{P}\Delta(0;1/n)$ and $r=(r_1,r_2,\ldots,r_n)$ such that \(\mathbf{r}=\|z\|_{\infty}\). Then
	\begin{align*}
	|f(z)|+|Df(z)|+\sum_{k=2}^{\infty}\sum_{|\alpha|=k}|a_{\alpha}|r^{\alpha} +\left(\frac{1}{1+|a_{0}|}+\frac{\mathbf{r}}{1-\mathbf{r}}\right)\sum_{k=1}^{\infty}\sum_{|\alpha|=k}|a_{\alpha}|^{2}r^{2\alpha}\leq 1
	\end{align*}
	for $n\mathbf{r}\leq {(\sqrt{17}-3)}/{4}$ and the constant $	{(\sqrt{17}-3)}/{4}$ is best possible. Moreover,
	
	\begin{align*}
		|f(z)|^2+|Df(z)|+\sum_{k=2}^{\infty}\sum_{|\alpha|=k}|a_{\alpha}|r^{\alpha} +\left(\frac{1}{1+|a_{0}|}+\frac{\mathbf{r}}{1-\mathbf{r}}\right)\sum_{k=1}^{\infty}\sum_{|\alpha|=k}|a_{\alpha}|^{2}r^{2\alpha}\leq 1
	\end{align*}
	for $\mathbf{r}\leq r_0$, where $nr_0\approx0.385795$ is the unique positive root of the equation 
	\begin{align*}
		1-2\mathbf{r}-(n\mathbf{r})^2-(n\mathbf{r})^3-(n\mathbf{r})^4=0
	\end{align*}
	and $nr_0$ is best possible.
\end{thmD}
Motivated by Theorem D, one can naturally raise the following question.
\begin{prob}\label{pb-1.2}
	Can we establish an improved version of the theorem D?
\end{prob}

In this paper, our aim is to present affirmative answers to the Problems \ref{pb-1.1} and \ref{pb-1.2}. The organization of the paper is as follows: In Section 2, we present our main results. In Section 3, we state several key lemmas which play an essential role in proving the results of this paper. Finally, in Section 4, we present the proofs of our main results.\vspace{1.2mm}

Before stating our main results, we highlight the geometric and analytical significance of the quantities under consideration. For a holomorphic function $f(z) = \sum_{|\alpha|=0}^{\infty} a_{\alpha} z^{\alpha}$ on the unit polydisk $\mathbb{D}^n$, the Euler operator $D f(z)$ acts as the total radial derivative. Observe that the $L^2$-norm of $Df$ over the boundary torus $\mathbb{T}^n_r = \{z \in \mathbb{C}^n : |z_j| = r, \, j=1,\dots,n\}$ yields
\[
\frac{1}{(2\pi)^n} \int_{0}^{2\pi} \dots \int_{0}^{2\pi} \left| Df(r e^{i\theta_1}, \dots, r e^{i\theta_n}) \right|^2 d\theta_1 \dots d\theta_n = \sum_{k=1}^{\infty} k \sum_{|\alpha|=k} |a_{\alpha}|^2 r^{2|\alpha|}.
\]
In our context, the lower-order weighted term 
\[
\mathcal{S}_f(r) := \sum_{k=1}^{\infty} k \sum_{|\alpha|=k} |a_{\alpha}|^2 r^{2\alpha}
\]
serves as a natural multidimensional analogue of the classical area measure $S_r/\pi$ of the Riemann surface $f(\mathbb{D}_r)$ in $\mathbb{C}$. It reflects the Dirichlet energy of $f$ restricted to polydisk slices.\vspace{1.2mm} 

By incorporating both the radial derivative operator $|Df(z)|$ and the area-type functional $\mathcal{S}_f(r)$, the majorant series studied in Theorems \ref{thm-2.1}--\ref{thm-2.3} capture subtle boundary and interior growth properties of holomorphic mappings in $\mathbb{C}^n$. The appearance of the dimensional factor $1/n$ in the sharp radii $1/(3n)$ and $1/(5n)$ reflects the fundamental geometric distortion arising from the supremum norm $\|\cdot\|_\infty$ on polydisks.
\section{\bf Main results}In this section, we present our main results concerning multidimensional Bohr-type inequalities. For $n = 1$, these results reduce to known classical single-variable inequalities. Furthermore, all the obtained Bohr radii are sharp.\vspace{1.2mm}

Let $f(z)=\sum_{|\alpha|=0} a_{\alpha}z^{\alpha}$ be a holomorphic function in the polydisk $\mathbb{P}\Delta(0;1/n)$ such that $|f(z)|\leq 1$ for all $z\in \mathbb{P}\Delta(0;1/n)$ and define $M_f(|z|)=\sum_{|\alpha|=0}^{\infty} |a_{\alpha}| |z|^{\alpha}$.\vspace{1.2mm}

% Lead-in for Theorem 2.1
Our first result establishes a higher-dimensional extension of Theorem B for bounded holomorphic functions on the unit polydisk $\mathbb{P}\Delta(0;1_n)$, incorporating both squared coefficient terms and an area-type functional.
\begin{thm}\label{thm-2.1}
	Let $f(z) = \sum_{|\alpha|=0}^{\infty} a_{\alpha}z^{\alpha}$ be a holomorphic function in the polydisk $\mathbb{P}\Delta(0; 1_n)$ such that $|f(z)| \le 1$ for all $z \in \mathbb{P}\Delta(0; 1/n)$. Suppose $z = (z_1, \dots, z_n) \in \mathbb{P}\Delta(0; 1/n)$ and $r = (r_1, r_2, \dots, r_n)$ such that $\|z\|_{\infty} = \mathbf{r}$. Then
	\begin{align*}
		\mathcal{C}^1_f(r):=\sum_{k=0}^{\infty} \sum_{|\alpha|=k} |a_{\alpha}| r^{\alpha} &+ \left( \frac{1}{1 + |a_0|} + \frac{\mathbf{r}}{1 - \mathbf{r}} \right) \sum_{k=1}^{\infty} \sum_{|\alpha|=k} |a_{\alpha}|^2 r^{2\alpha}\\& + \frac{8}{9} \sum_{k=1}^{\infty} k \sum_{|\alpha|=k} |a_{\alpha}|^2 r^{2\alpha}\leq 1,
	\end{align*}
	for $\mathbf{r}\leq 1/(3n)$ and the numbers $1/(3n)$ and $8/9$ cannot be improved. Moreover, 
	\begin{align*}
		\mathcal{C}^2_f(r):=|a_0|^2+\sum_{k=1}^{\infty} \sum_{|\alpha|=k} |a_{\alpha}| r^{\alpha} +& \left( \frac{1}{1 + |a_0|} + \frac{\mathbf{r}}{1 - \mathbf{r}} \right) \sum_{k=1}^{\infty} \sum_{|\alpha|=k} |a_{\alpha}|^2 r^{2\alpha}\\ +& \frac{9}{8} \sum_{k=1}^{\infty} k \sum_{|\alpha|=k} |a_{\alpha}|^2 r^{2\alpha}\leq 1,
	\end{align*}
	for $\mathbf{r}\leq 1/(n(3-|a_0|))$ and the number $9/8$ cannot be improved.
\end{thm}
% Lead-in for Theorem 2.2
Next, we refine Theorem D by introducing an additional area-type functional alongside the Euler operator $Df(z)$, providing an improved multidimensional Bohr-type inequality.
\begin{thm}\label{thm-2.2}
	Let $f(z)=\sum_{|\alpha|=0}^{\infty}a_{\alpha}z^{\alpha}$ be a holomorphic function in the polydisk $\mathbb{P}\Delta(0;1_n)$ such that $|f(z)|\leq 1$ for all \(z\in\mathbb{P}\Delta(0;1/n)\). Suppose $z=(z_{1},\ldots,z_{n})\in\mathbb{P}\Delta(0;1/n)$ and $r=(r_1,r_2,\ldots,r_n)$ such that \(\mathbf{r}=\|z\|_{\infty}\). Then
	\begin{align*}
		\mathcal{D}(z,\mathbf{r}):=&|f(z)|+|Df(z)|+\sum_{k=2}^{\infty}\sum_{|\alpha|=k}|a_{\alpha}|r^{\alpha} +\left(\frac{1}{1+|a_{0}|}+\frac{\mathbf{r}}{1-\mathbf{r}}\right)\sum_{k=1}^{\infty}\sum_{|\alpha|=k}|a_{\alpha}|^{2}r^{2\alpha}\\&+\frac{1}{4}(1-\mathbf{r}_*)(2+5\mathbf{r}_*+5\mathbf{r}^2_*)\sum_{k=1}^{\infty} k \sum_{|\alpha|=k} |a_{\alpha}|^2 r^{2\alpha}\leq 1
	\end{align*}
	for $n\mathbf{r}\leq\mathbf{r}_*={(\sqrt{17}-3)}/{4}$ and the constants $	\mathbf{r}_*$ and $\frac{1}{4}(1-\mathbf{r}_*)(2+5\mathbf{r}_*+5\mathbf{r}^2_*)$ are best possible. Moreover,
	\begin{align*}
		\mathcal{E}(z,\mathbf{r}):=&|f(z)|^2+|Df(z)|+\sum_{k=2}^{\infty}\sum_{|\alpha|=k}|a_{\alpha}|r^{\alpha} +\left(\frac{1}{1+|a_{0}|}+\frac{\mathbf{r}}{1-\mathbf{r}}\right)\sum_{k=1}^{\infty}\sum_{|\alpha|=k}|a_{\alpha}|^{2}r^{2\alpha}\\&+\mathbf{r}_{**}(1-\mathbf{r}^2_{**})\sum_{k=1}^{\infty} k \sum_{|\alpha|=k} |a_{\alpha}|^2 r^{2\alpha}\leq 1
	\end{align*}
	for $n\mathbf{r}\leq \mathbf{r}_{**}$, where $\mathbf{r}_{**}\approx0.385795$ is the unique positive root of the equation $1-2\mathbf{r}-\mathbf{r}^2-\mathbf{r}^3-\mathbf{r}^4=0$,
	and the constants $\mathbf{r}_{**}$ and $\mathbf{r}_{**}(1-\mathbf{r}^2_{**})$ are best possible.
\end{thm}
% Lead-in for Theorem 2.3
Finally, we extend Theorem C to several complex variables by incorporating the distance term $|f(z)-a_0|$ into the multidimensional majorant series.
\begin{thm}\label{thm-2.3}
	Let $f(z) = \sum_{|\alpha|=0}^{\infty} a_{\alpha}z^{\alpha}$ be a holomorphic function in the polydisk $\mathbb{P}\Delta(0; 1_n)$ such that $|f(z)| \le 1$ for all $z \in \mathbb{P}\Delta(0; 1/n)$. Suppose $z = (z_1, \dots, z_n) \in \mathbb{P}\Delta(0; 1/n)$ and $r = (r_1, r_2, \dots, r_n)$ such that $\|z\|_{\infty} = \mathbf{r}$. Then
	\begin{align*}
		\mathcal{F}^1_f(r):=\sum_{k=0}^{\infty} \sum_{|\alpha|=k} |a_{\alpha}| r^{\alpha} + \left( \frac{1}{1 + |a_0|} + \frac{\mathbf{r}}{1 - \mathbf{r}} \right) \sum_{k=1}^{\infty} \sum_{|\alpha|=k} |a_{\alpha}|^2 r^{2\alpha} + |f(z)-a_0|\leq 1,
	\end{align*}
	for $\mathbf{r}\leq 1/(5n)$ and the number $1/(5n)$ cannot be improved. Moreover, 
	\begin{align*}
		\mathcal{F}^2_f(r):=|a_0|^2+\sum_{k=1}^{\infty} \sum_{|\alpha|=k} |a_{\alpha}| r^{\alpha}& + \left( \frac{1}{1 + |a_0|} + \frac{\mathbf{r}}{1 - \mathbf{r}} \right) \sum_{k=1}^{\infty} \sum_{|\alpha|=k} |a_{\alpha}|^2 r^{2\alpha}\\& +|f(z)-a_0|\leq 1,
	\end{align*}
	for $\mathbf{r}\leq 1/(3n)$ and the number $1/(3n)$ cannot be improved.
\end{thm}
\section{\bf Key Lemmas}
We begin by reviewing a set of preliminary lemmas essential for developing our main results. The first of these is a specific formulation of Theorem 2.2 established in \cite{Chen-Hamada-Ponnusamy-Vijayakumar-JAM-2024}.
 \begin{lemA}
	Let $f$ be holomorphic in the polydisk $\mathbb{P}\Delta(0; 1_n)$ such that $|f(z)| \le 1$ for all $z \in \mathbb{P}\Delta(0; 1_n)$. Then for all $z \in \mathbb{P}\Delta(0; 1_n)$, we have
	\[
	|f(z)| \le \frac{|f(0)| + \|z\|_\infty}{1 + |f(0)| \|z\|_\infty}.
	\]
\end{lemA}
We shall make use of the following lemma, which can be found in \cite[Corollary 1.3]{Chen-Liu-IJPAM-2012}.
\begin{lemB}
	Let $f$ be holomorphic in the polydisk $\mathbb{P}\Delta(0; 1_n)$ such that $|f(z)| < 1$ for all $z \in \mathbb{P}\Delta(0; 1_n)$. Then for any multi-index $\alpha = (\alpha_1, \dots, \alpha_n)$, we have
	\[
	\left| \frac{\partial^{|\alpha|} f(z)}{\partial z_1^{\alpha_1} \dots \partial z_n^{\alpha_n}} \right| \le \frac{\alpha! \left(1 - |f(z)|^2\right)}{\left(1 - \|z\|_\infty^2\right)^{|\alpha|} }\left(1 + \|z\|_\infty\right)^{|\alpha|-N}
	\]
	for all $z \in \mathbb{P}\Delta(0; 1_n)$, where $N$ is the number of indices $j$ such that $\alpha_j \neq 0$.
\end{lemB}
\noindent Recently, Ahamed \emph{et al.} \cite{Ahamed-Majumder-Liu-Sarkar-2025} have established the following lemma.
 
\begin{lemC}\cite[Lemma 3.1]{Ahamed-Majumder-Liu-Sarkar-2025}
	Let $f$ be a holomorphic function in the polydisk $\mathbb{P}\Delta(0; 1_n)$ such that $|f(z)| \le 1$ for all $z \in \mathbb{P}\Delta(0; 1_n)$ and $f(z) = \sum_{|\alpha|=0}^\infty a_\alpha z^\alpha$ for all $z \in \mathbb{P}\Delta(0; 1_n)$. Then, for $r = (r_1, r_2, \dots, r_n)$ and $\mathbf{r} = \|r\|_\infty$, we have the estimates:
	
	\begin{enumerate}
		\item[(a)] 
		\begin{align}\label{eq-1.2}
			\sum_{k=1}^\infty k \sum_{|\alpha|=k} |a_\alpha|^2 \mathbf{r}^{2|\alpha|} \le \frac{\mathbf{r}^2 (1 - |a_0|^2)^2}{(1 - |a_0|^2 \mathbf{r}^2)^2}, \quad \text{for } 0 < \mathbf{r} \le \frac{1}{\sqrt{2}}.
		\end{align}
		
		\item [(b)]
		\begin{align}\label{eq-1.3}
			\sum_{k=1}^\infty \sum_{|\alpha|=k} |a_\alpha|^2 \mathbf{r}^{|\alpha|} \le \frac{\mathbf{r} (1 - |a_0|^2)^2}{1 - |a_0|^2 \mathbf{r}}, \quad \text{for } 0 < \mathbf{r} < 1.
		\end{align}
		
		\item [(c)]
		
		\begin{align*}
			\sum_{k=1}^\infty \sum_{|\alpha|=k} |a_\alpha| r^\alpha\leq\sum_{k=1}^\infty \sum_{|\alpha|=k} |a_\alpha| \mathbf{r}^{|\alpha|}\leq
			\begin{cases}
				\dfrac{\sqrt{n} \mathbf{r} (1 - |a_0|^2)}{1 - n|a_0|\mathbf{r}}, & \text{for } |a_0| \ge \mathbf{r}, \\[12pt]
				\dfrac{\sqrt{n} \mathbf{r} \sqrt{1 - |a_0|^2}}{\sqrt{1 - n \mathbf{r}^2}}, & \text{for } |a_0| < \mathbf{r}.
			\end{cases}
		\end{align*}
		
	\end{enumerate}
	The inequalities \eqref{eq-1.2} and \eqref{eq-1.3} are sharp.
\end{lemC} 
\begin{lemD}\cite[Lemma 3.5]{Ahamed-Majumder-Sarkar-CAOT-2026}
	Let $f$ be a holomorphic function in the polydisk $\mathbb{P}\Delta(0; 1_n)$ such that $|f(z)| \le 1$ and $f(z) = \sum_{|\alpha|=0}^\infty a_\alpha z^\alpha$ for all $z \in \mathbb{P}\Delta(0; 1_n)$. Then for any $N \in \mathbb{N}$, the following sharp inequality holds:
	\begin{align*}
		\sum_{k=N}^\infty \sum_{|\alpha|=k} |a_\alpha| r^\alpha +& \operatorname{sgn}(t) \sum_{k=1}^t \sum_{|\alpha|=k} |a_\alpha|^2 \frac{\mathbf{r}^N}{1 - \mathbf{r}} + \left( \frac{1}{1 + |a_0|} + \frac{\mathbf{r}}{1 - \mathbf{r}} \right) \sum_{k=t+1}^\infty \sum_{|\alpha|=k} |a_\alpha|^2 r^{2\alpha}\\& \le \frac{(1 - |a_0|^2)(n\mathbf{r})^N}{1 - n\mathbf{r}}
	\end{align*}
	for $n\mathbf{r} \in [0, 1)$, where $r = (r_1, r_2, \dots, r_n)$, $\mathbf{r} = \|r\|_\infty$, and $t = \left[\frac{N-1}{2}\right]$, with $[x]$ denoting the largest integer no more than $x$ for a real number $x$.
\end{lemD}
\section{\bf Proof of the Main Results}
\begin{proof}[\bf Proof of Theorem \ref{thm-2.1}]
	Since $f(z)=a_0+\sum_{|\alpha|=1}^{\infty}a_{\alpha}z^{\alpha}$
	is holomorphic in $\mathbb{P}\Delta(0;1_n)$ and satisfies the inequality $|f(z)|\leq1$ in $\mathbb{P}\Delta(0;1_n)$, we may choose $z=(z_1,\ldots,z_n)\in \mathbb{P}\Delta(0;1_n)$ with $\mathbf{r}=\|z\|_{\infty}$.\vspace{1.2mm}
	
	By Lemma C (a), we see that
	\begin{align}\label{eq-2.1}
		\sum_{k=1}^{\infty} k \sum_{|\alpha|=k} |a_{\alpha}|^2 r^{2\alpha} \le \sum_{k=1}^{\infty} k \sum_{|\alpha|=k} |a_{\alpha}|^2 \mathbf{r}^{2|\alpha|} \le \frac{\mathbf{r}^2 (1 - |a_0|^2)^2}{(1 - |a_0|^2 \mathbf{r}^2)^2}, 
	\end{align}
	for $0 < \mathbf{r} \le 1/\sqrt{2}$.\vspace{1.2mm}
	
	Since $|a_0|r^2 \le |a_0|n^2r^2$, it follows from \eqref{eq-2.1} that
	\begin{align}\label{eq-2.2}
		\frac{\mathbf{r}^2 (1 - |a_0|^2)^2}{(1 - |a_0|^2 \mathbf{r}^2)^2} \le \frac{n^2\mathbf{r}^2 (1 - |a_0|^2)^2}{(1 - |a_0|^2 n^2\mathbf{r}^2)^2}.
	\end{align}
	Using the Lemma D with $N=1$, we have 
	\begin{align}\label{eq-2.3}
		\sum_{k=1}^{\infty} \sum_{|\alpha|=k} |a_{\alpha}| r^{\alpha} + \left( \frac{1}{1 + |a_0|} + \frac{\mathbf{r}}{1 - \mathbf{r}} \right) \sum_{k=1}^{\infty} \sum_{|\alpha|=k} |a_{\alpha}|^2 r^{2\alpha} \le \frac{(1 - |a_0|^2) n\mathbf{r}}{1 - n\mathbf{r}}.
	\end{align}
From the inequalities \eqref{eq-2.1}, \eqref{eq-2.2}, and \eqref{eq-2.3}, we obtain
\begin{align}\label{eq-2.4}
	\mathcal{C}^1_f(r)\leq|a_0|+\frac{(1 - |a_0|^2) n\mathbf{r}}{1 - n\mathbf{r}}+\frac{8}{9}\frac{n^2\mathbf{r}^2 (1 - |a_0|^2)^2}{(1 - |a_0|^2 n^2\mathbf{r}^2)^2}:=P_1(\mathbf{r}).
\end{align}
It  can be easily seen that  $P_1$ is an increasing function of $\mathbf{r}$ and hence we obtain
\begin{align}\label{eq-2.5}
	P_1(\mathbf{r})\leq& P_1\left(\frac{1}{3n}\right)\\=&|a_0|+\frac{1-|a_0|^2}{2}+\frac{8(1-|a_0|^2)^2}{(9-|a_0|^2)^2}\nonumber\\=&1-\frac{(1-|a_0|)^2}{2(9-|a_0|^2)^2}\left((9-|a_0|^2)^2-16(1+|a_0|)^2\right)\nonumber\\=&1-\frac{(1-|a_0|)^3(|a_0|+5)}{2(9-|a_0|^2)^2}\phi_1(|a_0|)\nonumber,
\end{align}
where $\phi_1(x)=13+4x-x^2$, for $x\in[0,1]$. Clearly, $\phi_1$ is an increasing function of $x$, since $\phi^{\prime}_1(x) = 2(2-x) > 0$ for all $x \in [0,1]$. This implies that $\phi_1(x) \ge \phi_1(0) = 13 > 0$ for all $x \in [0,1]$. It then follows from \eqref{eq-2.5} that $P_1(\mathbf{r}) \le 1$. Consequently, \eqref{eq-2.4} yields $\mathcal{C}^1_f(r) \le 1$ for $\mathbf{r}\leq 1/(3n)$.\vspace{2mm}

To prove the sharpness part of the result, we consider the function
\begin{align}\label{eq-2.6}
	f_a(z) = \frac{a + (z_1 + z_2 + \dots + z_n)}{1 + a(z_1 + z_2 + \dots + z_n)}, 
\end{align}
where $a \in [0, 1)$.\vspace{1.2mm}

Clearly, $f$ is holomorphic in $\mathbb{P}\Delta(0; 1/n)$. Since $|a(z_1 + z_2 + \dots + z_n)| < 1$, we have
\begin{align*}
	f_a(z) = a + (1 - a^2) \sum_{k=1}^{\infty} (-a)^{k-1}(z_1 + z_2 + \dots + z_n)^k
\end{align*}
for all $z \in \mathbb{P}\Delta(0; 1/n)$.\vspace{1.2mm} 

For the point $z = (r, r, \dots, r)$, we find that
\begin{align*}
	f(z) = a +(1 - a^2) \sum_{k=1}^{\infty} (-a)^{k-1}(n\mathbf{r})^k. 
\end{align*}
If we take $a_0 = a$ and $a_k = (1 - a^2)(-a)^{k-1}$, then it to see that 
\begin{align*}
	&\sum_{k=1}^{\infty} \sum_{|\alpha|=k} |a_{\alpha}| r^{\alpha} + \left( \frac{1}{1 + |a_0|} + \frac{\mathbf{r}}{1 - \mathbf{r}} \right) \sum_{k=1}^{\infty} \sum_{|\alpha|=k} |a_{\alpha}|^2 r^{2\alpha}+\lambda\sum_{k=1}^{\infty}k\sum_{|\alpha|=k} |a_{\alpha}|^2 r^{2\alpha}\\
	&= \sum_{k=1}^{\infty} |a_k| (n\mathbf{r})^k + \left( \frac{1}{1 + |a_0|} + \frac{n\mathbf{r}}{1 - n\mathbf{r}} \right) \sum_{k=1}^{\infty} |a_k|^2 (n\mathbf{r})^{2k}+\lambda\sum_{k=1}^{\infty} k|a_k|^2 (n\mathbf{r})^{2k}. 
\end{align*}
Consequently, in view of the above, we obtain  that
\begin{align*}
	\mathcal{C}^1_{f_a}(r)=&a+\sum_{k=1}^{\infty} |a_k| (n\mathbf{r})^k + \left( \frac{1}{1 + |a_0|} + \frac{n\mathbf{r}}{1 - n\mathbf{r}} \right) \sum_{k=1}^{\infty} |a_k|^2 (n\mathbf{r})^{2k}+\lambda\sum_{k=1}^{\infty} k|a_k|^2 (n\mathbf{r})^{2k}\\=&a+\frac{(1-a^2)n\mathbf{r}}{1-an\mathbf{r}}+\frac{1+an\mathbf{r}}{(1-n\mathbf{r})(1+a)}\frac{(1-a^2)^2n^2\mathbf{r}^2}{1-a^2n^2\mathbf{r}^2}+\lambda\frac{(1-a^2)^2n^2\mathbf{r}^2}{(1-a^2n^2\mathbf{r}^2)^2}\\=&a+\frac{(1-a^2)n\mathbf{r}}{1-n\mathbf{r}}+\lambda\frac{(1-a^2)^2n^2\mathbf{r}^2}{(1-a^2n^2\mathbf{r}^2)^2}.
\end{align*}
For $\mathbf{r}=1/(3n)$, we have 
\begin{align*}
\mathcal{C}^1_{f_a}\left(\frac{1}{3n}\right)=&a+\frac{1-a^2}{2}+\frac{9\lambda(1-a^2)^2}{(9-a^2)^2}\\=&1+(1-a)^2\left(-\frac{1}{2}+\frac{9\lambda(1+a)^2}{(9-a^2)^2}\right),
\end{align*}
which is seen to be bigger than 1 when $\lambda>8/9$ as $a\to1$. \vspace{2mm}

We turn our attention now to the second part of the proof. From the inequalities \eqref{eq-2.1}, \eqref{eq-2.2}, and \eqref{eq-2.3}, we find that 
\begin{align}\label{eq-2.7}
	\mathcal{C}^2_f(r)\leq|a_0|^2+\frac{(1 - |a_0|^2) n\mathbf{r}}{1 - n\mathbf{r}}+\frac{9}{8}\frac{n^2\mathbf{r}^2 (1 - |a_0|^2)^2}{(1 - |a_0|^2 n^2\mathbf{r}^2)^2}:=P_2(\mathbf{r}).
\end{align}
One can easily verify that $P_2$ is an increasing function of $\mathbf{r}$, from which we obtain
\begin{align}\label{eq-2.8}
	P_2(\mathbf{r})\leq& P_2\left(\frac{1}{n(3-|a_0|)}\right)\\=&|a_0|^2+\frac{1-|a_0|^2}{2-|a_0|}+\frac{(1-|a_0|^2)^2(3-|a_0|)^2}{8(3-2|a_0|)^2}\nonumber\\=&1-\frac{(1-|a_0|)^2(1+|a_0|)}{8(3-2|a_0|)^2(2-|a_0|)}\left(8(3-2|a_0|)^2-(1+|a_0|)(3-|a_0|)^2(2-|a_0|)\right)\nonumber\\=&1-\frac{(1-|a_0|)^2(1+|a_0|)}{8(3-2|a_0|)^2(2-|a_0|)}\phi_2(|a_0|)\nonumber,
\end{align}
where $\phi_2(x)=54-39x+6x^2-x^3$, for $x\in[0,1]$. Evidently, $\phi_2(x)$ is strictly decreasing on $[0,1]$ because its derivative satisfies $\phi^{\prime}_2(x) = -3((x-2)^2+9) < 0$ throughout this domain. As a result, $\phi_2(x) \ge \phi_2(1) = 20 > 0$ for every $x \in [0,1]$. By applying \eqref{eq-2.8}, we deduce that $P_2(\mathbf{r}) \le 1$, from which \eqref{eq-2.7} guarantees $\mathcal{C}^2_f(r) \le 1$ whenever $\mathbf{r}\leq 1/(n(3-|a_0|))$.\vspace{2mm}

To show the sharpness, consider the holomorphic function $f$ defined on the polydisk $\mathbb{P}\Delta(0; 1/n)$ by \eqref{eq-2.6}. Setting $z = (r, r, \dots, r)$, we obtain
\begin{align*}
	\mathcal{C}^2_{f_a}(r)=&a^2+\sum_{k=1}^{\infty} |a_k| (n\mathbf{r})^k + \left( \frac{1}{1 + |a_0|} + \frac{n\mathbf{r}}{1 - n\mathbf{r}} \right) \sum_{k=1}^{\infty} |a_k|^2 (n\mathbf{r})^{2k}+\lambda\sum_{k=1}^{\infty} k|a_k|^2 (n\mathbf{r})^{2k}\\=&a^2+\frac{(1-a^2)n\mathbf{r}}{1-an\mathbf{r}}+\frac{1+an\mathbf{r}}{(1-n\mathbf{r})(1+a)}\frac{(1-a^2)^2n^2\mathbf{r}^2}{1-a^2n^2\mathbf{r}^2}+\lambda\frac{(1-a^2)^2n^2\mathbf{r}^2}{(1-a^2n^2\mathbf{r}^2)^2}\\=&a^2+\frac{(1-a^2)n\mathbf{r}}{1-n\mathbf{r}}+\lambda\frac{(1-a^2)^2n^2\mathbf{r}^2}{(1-a^2n^2\mathbf{r}^2)^2}.
\end{align*}
For $\mathbf{r}=1/(n(3-a))$, we have 
\begin{align*}
	\mathcal{C}^2_{f_a}\left(\frac{1}{n(3-a)}\right)=&1 + \frac{(a-1)^2(a+1)}{9(2-a)(3-2a)^2}\left(\lambda (a+1)(3-a)^2(2-a) - 9(3-2a)^2\right),
\end{align*}
which is seen to be bigger than 1 when $\lambda>9/8$ as $a\to1$.
\end{proof}

 \begin{proof}[\bf Proof of Theorem \ref{thm-2.2}]
 	Let $z=(z_1,\ldots,z_n)\in\mathbb{P}\Delta(0;1_n)$ be such that $\mathbf{r}=\|z\|_{\infty}$. It follows from Lemma A that
 	\begin{align}\label{eq-2.9}
 		|f(z)| \le \frac{\|z\|_\infty + |a_0|}{1 + |a_0|\|z\|_\infty} = \frac{\mathbf{r} + |a_0|}{1 + |a_0|\mathbf{r}} \le \frac{n\mathbf{r} + |a_0|}{1 + |a_0|n\mathbf{r}}.
 	\end{align}
 	Applying Lemma B yields the inequality
 	\begin{align}\label{eq-2.10}
 		|Df(z)|\leq \frac{1-|f(z)|^2}{1-\mathbf{r}^2}n\mathbf{r} \leq \frac{1-|f(z)|^2}{1-(n\mathbf{r})^2}n\mathbf{r}.
 	\end{align}
 	The function $\Psi(x)=x+\alpha(1-x^2)$ is increasing on the interval \([0,1]\) for \(0\leq\alpha\leq\frac{1}{2}\). Hence, $\Psi(x)\leq \Psi(x_0),\quad 0\leq x\leq x_0(\leq 1).$ Taking  
 	\begin{align*}
 	x=|f(z)|,\;\;	x_0=\dfrac{n\mathbf{r}+|a_0|}{1+n\mathbf{r}|a_0|},\;\; \mbox{and}\;\; \alpha=\dfrac{n\mathbf{r}}{1-(n\mathbf{r})^2},
 	\end{align*} we obtain
 	\begin{align}\label{eq-2.11}
 		|f(z)|+\frac{n\mathbf{r}}{1-(n\mathbf{r})^2}(1-|f(z)|^2)\leq\frac{n\mathbf{r}+|a_0|}{1+n\mathbf{r}|a_0|}+\frac{n\mathbf{r}}{1-(n\mathbf{r})^2}\left(1-\left(\frac{n\mathbf{r}+|a_0|}{1+n\mathbf{r}|a_0|}\right)^2\right),
 	\end{align}
 	for $0\leq n\mathbf{r}\leq \sqrt{2}-1$.
Lemma D (for $N=2$), together with estimates \eqref{eq-2.1}, \eqref{eq-2.2}, \eqref{eq-2.9}, \eqref{eq-2.10}, and \eqref{eq-2.11}, leads to
\begin{align}\label{eq-2.12}
	\mathcal{D}(z,\mathbf{r})\leq&|f(z)|+\frac{n\mathbf{r}}{1-(n\mathbf{r})^2}(1-|f(z)|^2)+\frac{(1 - |a_0|^2)( n\mathbf{r})^2}{1 - n\mathbf{r}}\\\nonumber&+\frac{1}{4}(1-\mathbf{r}_*)(2+5\mathbf{r}_*+5\mathbf{r}^2_*)\frac{n^2\mathbf{r}^2 (1 - |a_0|^2)^2}{(1 - |a_0|^2 n^2\mathbf{r}^2)^2}\\\nonumber=&\frac{n\mathbf{r}+|a_0|}{1+n\mathbf{r}|a_0|}+\frac{n\mathbf{r}}{1-(n\mathbf{r})^2}\left(1-\left(\frac{n\mathbf{r}+|a_0|}{1+n\mathbf{r}|a_0|}\right)^2\right)+\frac{(1 - |a_0|^2)( n\mathbf{r})^2}{1 - n\mathbf{r}}\\\nonumber&+\frac{1}{4}(1-\mathbf{r}_*)(2+5\mathbf{r}_*+5\mathbf{r}^2_*)\frac{n^2\mathbf{r}^2 (1 - |a_0|^2)^2}{(1 - |a_0|^2 n^2\mathbf{r}^2)^2}\\\nonumber=&\frac{n\mathbf{r}+|a_0|}{1+n\mathbf{r}|a_0|}+\frac{(1-|a_0|^2)n\mathbf{r}}{(1+|a_0|n\mathbf{r})^2}+\frac{(1 - |a_0|^2)( n\mathbf{r})^2}{1 - n\mathbf{r}}+\frac{1}{4}(1-\mathbf{r}_*)(2+5\mathbf{r}_*+5\mathbf{r}^2_*)\nonumber\\&\times\frac{n^2\mathbf{r}^2 (1 - |a_0|^2)^2}{(1 - |a_0|^2 n^2\mathbf{r}^2)^2}=:P_3(\mathbf{r})\nonumber.
\end{align}
One can readily observe that $P_3$ is an increasing function of $\mathbf{r}$. Consequently,
\begin{align}\label{eq-2.13}
	P_3(\mathbf{r})&\leq P_3\left(\frac{\mathbf{r}_*}{n}\right)\\\nonumber&=\frac{\mathbf{r}_*+|a_0|}{1+\mathbf{r}_*|a_0|}+\frac{(1-|a_0|^2)\mathbf{r}_*}{(1+|a_0|\mathbf{r}_*)^2}+\frac{(1 - |a_0|^2)\mathbf{r}^2_*}{1 - \mathbf{r}_*}+\frac{1}{4}(1-\mathbf{r}_*)(2+5\mathbf{r}_*+5\mathbf{r}^2_*)\frac{(1 - |a_0|^2)^2\mathbf{r}^2_*}{(1 - |a_0|^2 \mathbf{r}^2_*)^2}\\\nonumber&=1+(1-|a_0
	|)\bigg(-\frac{1-\mathbf{r}_*}{1+\mathbf{r}_*|a_0|}+\frac{(1+|a_0|)\mathbf{r}_*}{(1+|a_0|\mathbf{r}_*)^2}+\frac{(1+|a_0|)\mathbf{r}^2_*}{1 - \mathbf{r}_*}\\\nonumber&\quad+\frac{1}{4}(1-\mathbf{r}_*)(2+5\mathbf{r}_*+5\mathbf{r}^2_*)\frac{(1 - |a_0|^2)(1+|a_0|)\mathbf{r}^2_*}{(1 - |a_0|^2 \mathbf{r}^2_*)^2}\bigg)\\\nonumber&=1+(1-|a_0
	|)^2\bigg(-\frac{(|a_0|^2 + 2|a_0| + 2)\mathbf{r}_*^4 +(2|a_0| + 3)\mathbf{r}_*^3 +2\mathbf{r}_*^2}{(1 - \mathbf{r}_*)(1 + |a_0|\mathbf{r}_*)^2}\\\nonumber&\quad+\frac{1}{4}(1-\mathbf{r}_*)(2+5\mathbf{r}_*+5\mathbf{r}^2_*)\frac{(1+|a_0|)^2\mathbf{r}^2_*}{(1 - |a_0|^2 \mathbf{r}^2_*)^2}\bigg)\\\nonumber&=1+\frac{(1-|a_0
		|)^2\mathbf{r}_*^2}{4(1 - \mathbf{r}_*)(1 - |a_0|^2 \mathbf{r}^2_*)^2}\Phi_3(|a_0|),\nonumber
\end{align}
where \begin{align*}
	\Phi_3(|a_0|)=&-4(1 - |a_0|\mathbf{r}_*)^2\left((|a_0|^2 + 2|a_0| + 2)\mathbf{r}_*^2 +(2|a_0| + 3)\mathbf{r}_* +2\right)\\&\quad+(1+|a_0|)^2(1-\mathbf{r}_*)^2(2+5\mathbf{r}_*+5\mathbf{r}^2_*).
\end{align*}
We have
\begin{align*}
	\Phi^{\prime}_3(|a_0|)=&\left(-16|a_0|^3 - 24|a_0|^2 - 6|a_0| + 10\right)\mathbf{r}_*^4 + \left(6 - 2|a_0|\right)\mathbf{r}_*^3 + \left(2|a_0|+ 10\right)\mathbf{r}_*^2 \\&+ \left(2|a_0| + 10\right)\mathbf{r}_* + \left(4|a_0| + 4\right).
\end{align*}
One can readily verify that $\Phi^{\prime}_3(\vert{}a_0\vert{})\geq 0$ for all $\vert{}a_0\vert{}\in[0,1]$, establishing that $\Phi_3$ is an increasing function on $[0,1]$. It follows that $\Phi_3(\vert{}a_0\vert{})\leq \Phi_3(1)=0$ throughout this domain. Applying \eqref{eq-2.13} gives $P_3(\mathbf{r})\leq 1$, which via inequality \eqref{eq-2.12} leads to $\mathcal{D}(z,\mathbf{r})\leq 1$ provided $n\mathbf{r}\leq\mathbf{r}_*=(\sqrt{17}-3)/4$.\vspace{2mm}

 To establish the sharpness, consider the holomorphic function \(f_a\) defined on the polydisk \(\mathbb{P}\Delta(0;1/n)\) by \eqref{eq-2.6}. Taking \(z=(r,r,\dots,r)\), we obtain
 	\begin{align*}
 		\mathcal{D}_{f_a}(z,\mathbf{r})=&\frac{n\mathbf{r}+a}{1+an\mathbf{r}}+\frac{(1-a^2)n\mathbf{r}}{(1+an\mathbf{r})^2}+\frac{(1-a^2)a(n\mathbf{r})^2}{1-an\mathbf{r}}+\frac{1+an\mathbf{r}}{(1+a)(1-n\mathbf{r})}\left(\frac{(1-a^2)^2(n\mathbf{r})^2}{1-a^2(n\mathbf{r})^2}\right)\\&+\lambda\frac{(1-a^2)^2n^2\mathbf{r}^2}{(1-a^2n^2\mathbf{r}^2)^2}\\=&\frac{n\mathbf{r}+a}{1+an\mathbf{r}}+\frac{(1-a^2)n\mathbf{r}}{(1+an\mathbf{r})^2}+\frac{(1-a^2)(n\mathbf{r})^2}{1-n\mathbf{r}}+\lambda\frac{(1-a^2)^2n^2\mathbf{r}^2}{(1-a^2n^2\mathbf{r}^2)^2}, 
 	\end{align*}
 	which, upon setting $n\mathbf{r}=\mathbf{r}_*$, reduces to
 	\begin{align*}
 		\mathcal{D}_{f_a}\left(z,\frac{\mathbf{r}_*}{n}\right)=&1+(1-a)\bigg(-\frac{1-\mathbf{r}_*}{1+a\mathbf{r}_*}+\frac{(1+a)\mathbf{r}_*}{(1+a\mathbf{r}_*)^2}+\frac{(1+a)\mathbf{r}^2_*}{1-\mathbf{r}_*}+\lambda\frac{(1-a^2)(1+a)\mathbf{r}^2_*}{(1-a^2\mathbf{r}^2_*)^2}\bigg)\\=&1+\frac{(1-a)}{(1-\mathbf{r}_*)(1-a^2\mathbf{r}^2_*)^2}\bigg(-(1-\mathbf{r}_*)^2(1-a^2\mathbf{r}^2_*)(1-a\mathbf{r}_*)+(1+a)\mathbf{r}_*\\&\times(1-\mathbf{r}_*)(1-a\mathbf{r}_*)^2+(1+a)\mathbf{r}^2_*(1-a^2\mathbf{r}^2_*)^2+\lambda(1-a^2)(1+a)\mathbf{r}^2_*(1-\mathbf{r}_*)\bigg)\\=&1+\frac{(1-a)^2}{(1-\mathbf{r}_*)(1-a^2\mathbf{r}^2_*)^2}\bigg( \left[ -1 + 3\mathbf{r}_* + (2a - 1)\mathbf{r}_*^2 + (a + 2a^2)\mathbf{r}_*^3 + (a^2 + a^3)\mathbf{r}_*^4 \right]\\&\times\frac{(1 - a\mathbf{r}_*)^2}{(1-a)}+\lambda(1+a)^2\mathbf{r}^2_*(1-\mathbf{r}_*)\bigg)\\=&1+\frac{(1-a)^2}{(1-\mathbf{r}_*)(1-a^2\mathbf{r}^2_*)^2}\mathcal{R}_{\lambda,\mathbf{r}_*}(a),
 	\end{align*}
 	where 
 	\begin{align*}
 		\mathcal{R}_{\lambda,\mathbf{r}_*}(a)=&\bigg( \left[ -1 + 3\mathbf{r}_* + (2a - 1)\mathbf{r}_*^2 + (a + 2a^2)\mathbf{r}_*^3 + (a^2 + a^3)\mathbf{r}_*^4 \right]\frac{(1 - a\mathbf{r}_*)^2}{(1-a)}\\&+\lambda(1+a)^2\mathbf{r}^2_*(1-\mathbf{r}_*)\bigg)
 	\end{align*}
 	and hence 
 	\begin{align*}
 		\lim_{a\to1}\mathcal{R}_{\lambda,\mathbf{r}_*}(a)=-\mathbf{r}_*^2(1 - \mathbf{r}_*)^2(5\mathbf{r}_*^2 + 5\mathbf{r}_* + 2)+4\lambda\mathbf{r}^2_*(1-\mathbf{r}_*).
 	\end{align*}
 	It can be easily seen that 
 	\begin{align*}
 		\mathcal{D}_{f_a}\left(z,\frac{\mathbf{r}_*}{n}\right)>1\;\; \mbox{when}\;\; \lambda>\frac{1}{4}(1-\mathbf{r}_*)(2+5\mathbf{r}_*+5\mathbf{r}^2_*)\;\;\mbox{and}\;\;a\to1^{-}.
 	\end{align*}
 	 
 	It is straightforward to verify that for $0 \le x \le x_0 \le 1$ and $0 \le \alpha \le 1$, we have $\xi(x) := x^2 + \alpha(1-x^2) \le \xi(x_0)$. Furthermore, note that 
 	\begin{align*}
 		\dfrac{n\mathbf{r}}{1-(n\mathbf{r})^2} \le 1\;\; \mbox{whenever}\;\; n\mathbf{r} \le 1/2.
 	\end{align*}
 	Proceeding as in the previous case, we obtain
 	\begin{align}\label{eq-2.14}
 		\mathcal{E}(z,\mathbf{r})\leq&\left(\frac{n\mathbf{r}+|a_0|}{1+n\mathbf{r}|a_0|}\right)^2+\frac{(1-|a_0|^2)n\mathbf{r}}{(1+|a_0|n\mathbf{r})^2}+\frac{(1 - |a_0|^2)( n\mathbf{r})^2}{1 - n\mathbf{r}}+\mathbf{r}_{**}(1-\mathbf{r}^2_{**})\\&\times\frac{n^2\mathbf{r}^2 (1 - |a_0|^2)^2}{(1 - |a_0|^2 n^2\mathbf{r}^2)^2}=:P_4(\mathbf{r}).\nonumber
 	\end{align}
 	It is easy to see that $P_4$ is a monotonically increasing function of $\mathbf{r}$. As a result, we have
 	\begin{align}\label{eq-2.15}
 		P_4(\mathbf{r})&\leq P_4\left(\frac{\mathbf{r}_{**}}{n}\right)\\\nonumber&=\left(\frac{\mathbf{r}_{**}+|a_0|}{1+\mathbf{r}_{**}|a_0|}\right)^2+\frac{(1-|a_0|^2)\mathbf{r}_{**}}{(1+|a_0|\mathbf{r}_{**})^2}+\frac{(1 - |a_0|^2)\mathbf{r}^2_{**}}{1 - \mathbf{r}_{**}}+\mathbf{r}_{**}(1-\mathbf{r}^2_{**})\frac{(1 - |a_0|^2)^2\mathbf{r}^2_{**}}{(1 - |a_0|^2 \mathbf{r}^2_{**})^2}\\\nonumber&=1+(1-|a_0
 		|^2)\bigg(\frac{\mathbf{r}^2_{**}+\mathbf{r}_{**}-1}{(1+\mathbf{r}_{**}|a_0|)^2}+\frac{\mathbf{r}^2_{**}}{1 - \mathbf{r}_{**}}+\mathbf{r}_{**}(1-\mathbf{r}^2_{**})\frac{(1 - |a_0|^2)\mathbf{r}^2_{**}}{(1 - |a_0|^2 \mathbf{r}^2_{**})^2}\bigg)\\\nonumber&=1+(1-|a_0
 		|)^2(1+|a_0|)\bigg(-\frac{\mathbf{r}_{**}^3 ( (1 + |a_0|)\mathbf{r}_{**} + 2 )}{(1 - \mathbf{r}_{**})(1 + |a_0|\mathbf{r}_{**})^2}+\mathbf{r}_{**}(1-\mathbf{r}^2_{**})\frac{(1+|a_0|)\mathbf{r}^2_{**}}{(1 - |a_0|^2 \mathbf{r}^2_{**})^2}\bigg)\\\nonumber&=1+\frac{(1-|a_0
 			|)^2(1+|a_0|)\mathbf{r}_{**}^3}{(1 - \mathbf{r}_{**})(1 - |a_0|^2 \mathbf{r}^2_{**})^2}\Phi_4(|a_0|),\nonumber
 	\end{align}
 	where \begin{align*}
 		\Phi_4(|a_0|)=&-(1 - |a_0|\mathbf{r}_{**})^2( (1 + |a_0|)\mathbf{r}_{**} + 2 )+(1+|a_0|)(1-\mathbf{r}_{**})^2(1+\mathbf{r}_{**})
 	\end{align*}
 	and hence 
 	\begin{align*}
 		\Phi^{\prime}_4(|a_0|)=\mathbf{r}_{**}(1 - \vert{}a_0\vert{}\mathbf{r}_{**})\Big( 3\vert{}a_0\vert{}\mathbf{r}_{**} + 2\mathbf{r}_{**} + 3 \Big) + (1 - \mathbf{r}_{**})^2(1 + \mathbf{r}_{**})\geq 0
 	\end{align*}
 	for all $|a_0|\in[0,1]$.\vspace{1.2mm}
 	
  Consequently, $\Phi_4$ is monotonically increasing on $[0,1]$, which implies that $\Phi_4(|a_0|) \le \Phi_4(1) = 0$ for all $|a_0| \in [0,1]$. Applying \eqref{eq-2.15} yields $P_4(\mathbf{r}) \le 1$, which together with \eqref{eq-2.14} implies $\mathcal{E}(z,\mathbf{r}) \le 1$ for $n\mathbf{r} \le \mathbf{r}_{**} \approx 0.385795$.\vspace{2mm}
 	
 	To prove the sharpness, we consider the holomorphic function \(f_a\) defined on the polydisk \(\mathbb{P}\Delta(0;1/n)\) by \eqref{eq-2.6}. Choosing \(z=(r,r,\dots,r)\), we obtain
 	\begin{align*}
 		\mathcal{E}_{f_a}(z,\mathbf{r})=&\left(\frac{n\mathbf{r}+a}{1+an\mathbf{r}}\right)^2+\frac{(1-a^2)n\mathbf{r}}{(1+an\mathbf{r})^2}+\frac{(1-a^2)a(n\mathbf{r})^2}{1-an\mathbf{r}}\\&+\frac{1+an\mathbf{r}}{(1+a)(1-n\mathbf{r})}\left(\frac{(1-a^2)^2(n\mathbf{r})^2}{1-a^2(n\mathbf{r})^2}\right)+\lambda\frac{(1-a^2)^2n^2\mathbf{r}^2}{(1-a^2n^2\mathbf{r}^2)^2}\\=&\left(\frac{n\mathbf{r}+a}{1+an\mathbf{r}}\right)^2+\frac{(1-a^2)n\mathbf{r}}{(1+an\mathbf{r})^2}+\frac{(1-a^2)(n\mathbf{r})^2}{1-n\mathbf{r}}+\lambda\frac{(1-a^2)^2n^2\mathbf{r}^2}{(1-a^2n^2\mathbf{r}^2)^2}.
 	\end{align*}
 	Setting \(n\mathbf{r}=\mathbf{r}_{**}\), the above expression simplifies to
 	\begin{align*}
 		\mathcal{E}_{f_a}\left(z,\frac{\mathbf{r}_{**}}{n}\right)=&1+(1-a^2)\bigg(\frac{\mathbf{r}^2_{**}+\mathbf{r}_{**}-1}{(1+a\mathbf{r}_{**})^2}+\frac{\mathbf{r}^2_{**}}{1-\mathbf{r}_{**}}+\lambda\frac{(1-a^2)\mathbf{r}^2_{**}}{(1-a^2\mathbf{r}^2_{**})^2}\bigg)\\=&1+\frac{(1-a)^2(1+a)}{(1-\mathbf{r}_*)(1-a^2\mathbf{r}^2_*)^2}\bigg( \left[a^2\mathbf{r}_{**}^4 + (2a-1)\mathbf{r}_{**}^3 + \mathbf{r}_{**}^2 + 2\mathbf{r}_{**} - 1\right]\\&\times\frac{(1-a\mathbf{r}_{**})^2}{(1-a)}+\lambda(1+a)\mathbf{r}^2_{**}(1-\mathbf{r}_{**})\bigg)\\=&1+\frac{(1-a)^2}{(1-\mathbf{r}_*)(1-a^2\mathbf{r}^2_*)^2}\mathcal{Q}_{\lambda,\mathbf{r}_{**}}(a),
 	\end{align*}
 	where 
 	\begin{align*}
 		\mathcal{Q}_{\lambda,\mathbf{r}_{**}}(a)=&\bigg( \left[a^2\mathbf{r}_{**}^4 + (2a-1)\mathbf{r}_{**}^3 + \mathbf{r}_{**}^2 + 2\mathbf{r}_{**} - 1\right]\frac{(1-a\mathbf{r}_{**})^2}{(1-a)}\\&\quad+\lambda(1+a)\mathbf{r}^2_{**}(1-\mathbf{r}_{**})\bigg).
 	\end{align*}
 	It is easy to see that
 	\begin{align*}
 		\lim_{a\to1}\mathcal{Q}_{\lambda,\mathbf{r}_{**}}(a)=-2\mathbf{r}^3_{**}(1-\mathbf{r}_{**})^2(1+\mathbf{r}_{**})+2\lambda\mathbf{r}^2_{**}(1-\mathbf{r}_{**}).
 	\end{align*}
 	It can readily be seen that 
 	\begin{align*}
 		\mathcal{E}_{f_a}\left(z,\frac{\mathbf{r}_{**}}{n}\right)>1\;\; \mbox{when}\;\;\lambda>\mathbf{r}_{**}(1-\mathbf{r}^2_{**})\;\; \mbox{and}\;\; a\to1^{-}.
 	\end{align*}
 	This completes the proof.
 \end{proof}
 
 \begin{proof}[\bf Proof of Theorem \ref{thm-2.3}]
 	Given that $f(z) = a_0 + \sum_{|\alpha|=1}^{\infty} a_{\alpha} z^{\alpha}$ is holomorphic with $|f(z)| \le 1$ on $\mathbb{P}\Delta(0; 1_n)$, let $z = (z_1, \dots, z_n) \in \mathbb{P}\Delta(0; 1_n)$ be chosen with $\mathbf{r} = \|z\|_{\infty}$. Note that
 	
\begin{align}\label{eq-2.16}
	|f(z)-a_0|\leq \sum_{|\alpha|=1}^{\infty} |a_\alpha| |z|^\alpha \le \sum_{|\alpha|=1}^{\infty} |a_\alpha| \|z\|_\infty^{|\alpha|} = \sum_{|\alpha|=1}^{\infty} |a_\alpha| \mathbf{r}^{|\alpha|} =  \sum_{k=1}^{\infty} \sum_{|\alpha|=k} |a_\alpha| \mathbf{r}^{|\alpha|}.
\end{align}
 	By Lemma~C(c), it follows that
 	\begin{align}\label{eq-2.17}
 		\sum_{k=1}^\infty \sum_{|\alpha|=k} |a_\alpha| r^\alpha \le \sum_{k=1}^\infty \sum_{|\alpha|=k} |a_\alpha| \mathbf{r}^{|\alpha|} \le 
 		\begin{cases}
 			A(\mathbf{r}) := \dfrac{\sqrt{n} \mathbf{r} (1 - |a_0|^2)}{1 - n|a_0|r}, & \text{for } |a_0| \ge \mathbf{r}, \\[12pt]
 			B(\mathbf{r}) := \dfrac{\sqrt{n} \mathbf{r}\sqrt{1 - |a_0|^2}}{\sqrt{1 - n \mathbf{r}^2}}, & \text{for } |a_0| < \mathbf{r}.
 		\end{cases}
 	\end{align}
 	To establish the first part in full, it suffices to examine the following two cases:\vspace{2mm}
 	
 	\noindent\textbf{Case 1.} For $a = |a_0| \ge \mathbf{r}$ and $\mathbf{r}\leq 1/(5n) $, we deduce from \eqref{eq-2.17} that
 	\begin{align}\label{eq-2.18}
 		A(\mathbf{r}) := \dfrac{\sqrt{n} \mathbf{r} (1 - |a_0|^2)}{1 - n|a_0|\mathbf{r}} \le n\mathbf{r} \left( \frac{1 - |a_0|^2}{1 - n|a_0|\mathbf{r}} \right).
 	\end{align}
 	Combining the inequalities \eqref{eq-2.3}, \eqref{eq-2.16}, \eqref{eq-2.17}, and \eqref{eq-2.18}, we obtain
 	\begin{align*}
 		\mathcal{F}^1_f(r)\leq& a+\frac{(1-a^2)n\mathbf{r}}{1-n\mathbf{r}}+\frac{(1-a^2)n\mathbf{r}}{1-an\mathbf{r}}:=P_5(\mathbf{r})\\\leq& P_5\left(\frac{1}{5n}\right)\nonumber\\=&a+\frac{1-a^2}{4}+\frac{1-a^2}{5-a}\nonumber\\=&1-\frac{(1 - a)^2 (11-a)}{4(5 - a)}\nonumber\leq 1.
 	\end{align*}
 	\textbf{Case 2.} Assuming $a = |a_0| < \mathbf{r} \leq 1/(5n)$, equation \eqref{eq-2.17} implies that
 	\begin{align}\label{eq-2.19}
 		B(\mathbf{r}) := \dfrac{\sqrt{n} \mathbf{r}\sqrt{1 - |a_0|^2}}{\sqrt{1 - n \mathbf{r}^2}}\leq \dfrac{n \mathbf{r}\sqrt{1 - |a_0|^2}}{\sqrt{1 - n^2\mathbf{r}^2}}.
 	\end{align}
 	Applying inequalities \eqref{eq-2.3}, \eqref{eq-2.16}, \eqref{eq-2.17}, and \eqref{eq-2.19} gives
 	\begin{align*}
 		\mathcal{F}^1_f(r)\leq& a+\frac{(1-a^2)n\mathbf{r}}{1-n\mathbf{r}}+\dfrac{n \mathbf{r}\sqrt{1 - |a_0|^2}}{\sqrt{1 - n^2\mathbf{r}^2}}:=P_6(\mathbf{r})\\\leq& P_6\left(\frac{1}{5n}\right)\nonumber\\=&a + \frac{1 - a^2}{4} + \frac{\sqrt{1 - a^2}}{2\sqrt{6}}\\ \leq& \frac{1}{5n} + \frac{1}{4} + \frac{1}{4} < 1.
 	\end{align*}
 	To establish sharpness, we consider the holomorphic function $f_a$ defined on the polydisk $\mathbb{P}\Delta(0;1/n)$ by \eqref{eq-2.6}. Setting $z = (-r, -r, \dots, -r)$ yields
 	\begin{align*}
 		\mathcal{F}^1_{f_a}(r)=&a+\frac{(1-a^2)n\mathbf{r}}{1-n\mathbf{r}}+\frac{(1-a^2)n\mathbf{r}}{1-an\mathbf{r}}\\=&1-\frac{(1-a)}{(1-n\mathbf{r})(1-an\mathbf{r})}\Phi_5(a,\mathbf{r}),
 	\end{align*}
 	where $\Phi_5(a,\mathbf{r})=n^2 \mathbf{r}^2 a^2 + 3n\mathbf{r}(n\mathbf{r} - 1) a + (n^2 \mathbf{r}^2 - 3n\mathbf{r} + 1)$. \vspace{1.2mm}
 	
 	One readily observes that $\mathcal{F}^1_{f_a}(r)>1$ if and only if $\Phi_5(a,\mathbf{r})<0$. Assuming $\frac{1}{5n}<\mathbf{r}<\frac{3-\sqrt{5}}{2n}$ and defining $$a_{\mathbf{r}}= \frac{3(1 - n\mathbf{r}) - \sqrt{5n^2\mathbf{r}^2 - 6n\mathbf{r} + 5}}{2n\mathbf{r}},$$ a straightforward computation shows that $a_{\mathbf{r}}\in(0,1)$ and $\frac{3(1-n\mathbf{r})}{2n\mathbf{r}}>1$. Consequently, $\Phi_5(a,\mathbf{r})<\Phi_5(a_{\mathbf{r}},\mathbf{r})=0$ holds for all $a_{\mathbf{r}}<a<1$. This completes the sharpness part.\vspace{2mm}
 	
 	To establish the second part of the theorem, following a similar line of reasoning, it is sufficient to examine the following two cases:\vspace{2mm}
 	
 	\noindent\textbf{Case 1.} Assuming $a = |a_0| \ge \mathbf{r}$ and $\mathbf{r}\le 1/(3n)$, we arrive at inequality \eqref{eq-2.18}. From inequalities \eqref{eq-2.3}, \eqref{eq-2.16}, \eqref{eq-2.17}, and \eqref{eq-2.18}, we deduce that
 	\begin{align*}
 		\mathcal{F}^2_f(r)\leq& a^2+\frac{(1-a^2)n\mathbf{r}}{1-n\mathbf{r}}+\frac{(1-a^2)n\mathbf{r}}{1-an\mathbf{r}}:=P_7(\mathbf{r})\\\leq& P_7\left(\frac{1}{3n}\right)\nonumber\\=&a^2+\frac{1-a^2}{2}+\frac{1-a^2}{3-a}\nonumber\\=&1-{ \frac{(1-a)^2 (a + 1)}{2(3-a)} }\nonumber\leq 1.
 	\end{align*}
 	\textbf{Case 2.} Assuming $a = |a_0| < \mathbf{r} \le 1/(3n)$, we obtain inequality \eqref{eq-2.19}. By virtue of inequalities \eqref{eq-2.3}, \eqref{eq-2.16}, \eqref{eq-2.17}, and \eqref{eq-2.19}, it follows that
 	\begin{align*}
 		\mathcal{F}^2_f(r)\leq& a^2+\frac{(1-a^2)n\mathbf{r}}{1-n\mathbf{r}}+\dfrac{n \mathbf{r}\sqrt{1 - |a_0|^2}}{\sqrt{1 - n^2\mathbf{r}^2}}:=P_8(\mathbf{r})\\\leq& P_8\left(\frac{1}{3n}\right)\nonumber\\=&a^2 + \frac{1 - a^2}{2} + \frac{\sqrt{1 - a^2}}{2\sqrt{2}} \\\le& \frac{1}{9n^2} + \frac{1}{2} + \frac{\sqrt{2}}{4} < 1.
 	\end{align*}
 	In order to show sharpness, we inspect the holomorphic function $f_a$ defined on the polydisk $\mathbb{P}\Delta(0;1/n)$ by \eqref{eq-2.6}. At the point $z = (-r, -r, \dots, -r)$, we find
 	\begin{align*}
 		\mathcal{F}^2_{f_a}(r) &= a^2+\frac{(1-a^2)n\mathbf{r}}{1-n\mathbf{r}}+\frac{(1-a^2)n\mathbf{r}}{1-an\mathbf{r}}\\&=1-\frac{(1-a^2)}{(1-n\mathbf{r})(1-an\mathbf{r})}\Phi_6(a,\mathbf{r}),
 	\end{align*}
 	where $\Phi_6(a,\mathbf{r}) = { n\mathbf{r}(2n\mathbf{r} - 1)a + (n^2\mathbf{r}^2 - 3n\mathbf{r} + 1) }$.\vspace{1.2mm} 
 	
 	Clearly, $\mathcal{F}^2_{f_a}(r)>1$ if and only if $\Phi_6(a,\mathbf{r})<0$. Assuming $\frac{1}{3n}<\mathbf{r}<\frac{3-\sqrt{5}}{2n}$, let 
 	\begin{align*}
 		a_{\mathbf{r}}= \frac{n^2\mathbf{r}^2 - 3n\mathbf{r} + 1}{n\mathbf{r}(1-2n\mathbf{r})}.
 	\end{align*} 
 	A straightforward computation shows that $a_{\mathbf{r}}\in(0,1)$ and  $\Phi_6(a,\mathbf{r})<\Phi_6(a_{\mathbf{r}},\mathbf{r})=0$ for all $a_{\mathbf{r}}<a<1$. This completes the proof of sharpness.
 \end{proof}
  \section{\bf Conclusions}
  The paper formulates higher-dimensional analogues of Liu \textit{et al.}'s refined inequalities on the polydisk $\mathbb{P}\Delta(0; 1_n)$. It incorporates a weighted area functional 
  \begin{align*}
  	\mathcal{S}_f(r) := \sum_{k=1}^{\infty} k \sum_{|\alpha|=k} |a_{\alpha}|^2 r^{2\alpha},
  \end{align*} which represents the Dirichlet energy of $f$ restricted to polydisk slices. The authors derive a sharp Bohr radius of $\mathbf{r} \le 1/(3n)$ for the standard series and $\mathbf{r} \le 1/(n(3-|a_0|))$ when using the squared initial term $|a_0|^2$.\vspace{1.2mm}
  
  We strengthen existing multidimensional estimates by incorporating both the radial differential operator $Df(z) $ and the area functional $\mathcal{S}_f(r)$ within the majorant series. They establish best possible bounds for $n\mathbf{r} \le r_* = (\sqrt{17}-3)/4 \approx 0.2807$ and $n\mathbf{r} \le r_{**} \approx 0.385795$. \vspace{1.2mm}
  
  The single-variable refined Bohr inequality featuring the shift term $|f(z) - a_0|$ is extended to several complex variables, yielding optimal radii of $\mathbf{r} \le 1/(5n)$ and $\mathbf{r} \le 1/(3n)$ on $\mathbb{P}\Delta(0; 1_n)$.\vspace{1.2mm}
  
  All derived Bohr radii and constant coefficients are proved as sharp. The sharpness are shown using multi-variable polydisk extremal mappings of the form 
  \begin{align*}
  	f_a(z) = \frac{a + \sum_{j=1}^n z_j}{1 + a \sum_{j=1}^n z_j}\; \mbox{as}\;a \to 1^-.
  \end{align*}\vspace{1.2mm}
  
\noindent{\bf Acknowledgment:} The research of the first author is supported by Science and Engineering Research Board (SERB) (File No. SUR/2022/002244), Govt. of India, and the second author is supported by UGC-JRF (NTA Ref. No.: $ 221610103011$), New Delhi, India. \vspace{2mm}

%\section{Declaration}
\noindent\textbf{Compliance of Ethical Standards.}\\

\noindent\textbf{Conflict of interest.} The authors declare that there is no conflict  of interest regarding the publication of this paper.\vspace{1.5mm}

\noindent\textbf{Data availability statement.}  Data sharing not applicable to this article as no datasets were generated or analyzed during the current study.\vspace{1.5mm}

%\noindent\textbf{Funding.} No fund.

\end{document}